\documentclass[psamsfont,a4paper,12pt]{amsart}
\usepackage[english]{babel}   
\numberwithin{equation}{section}

\usepackage{setspace}     
\usepackage{enumerate} 	
\usepackage{afterpage} 	
\usepackage[usenames,dvipsnames]{color}
\usepackage{graphicx}	

\usepackage{amsmath} 	
\usepackage{amssymb}	
\usepackage{mathrsfs} 	
\usepackage{amsfonts} 	
\usepackage{latexsym} 	
\usepackage[latin1]{inputenc} 

\usepackage{amsthm} 	
\usepackage{stackrel}       
\usepackage{amscd} 	

\usepackage{tabularx}	
\usepackage{longtable}	

\usepackage{verbatim} 
\usepackage{blindtext} 

\usepackage{amssymb,latexsym,amsmath}    
\usepackage{pstricks,multido,pst-plot}		 
\usepackage{multicol,fancyheadings}			   
\usepackage{float}                       
\usepackage[vcentermath,enableskew]{youngtab}

\usepackage{subfig}

\allowdisplaybreaks

\def\F{{\mathbb F}}

\newcommand{\C}{\mathbb{C}}

\newcommand{\Z}{\mathbb{Z}} 
\newcommand{\Q}{\mathbb{Q}}

\newcommand{\FF}{\mathbb{F}}

\def\GL{\mathrm{GL}}
\def\SL{\mathrm{SL}}

\renewcommand{\to}{\longrightarrow}

\newtheorem{theorem}{Theorem}[section]  

\newtheorem{lemma}[theorem]{Lemma}

\newtheorem*{remark}{Remark}

\theoremstyle{definition}

\theoremstyle{remark}
 
\newtheorem*{rmk}{Remark}

\newtheorem{ind}[]{{\rm\it Indice}}

\usepackage{multicol}

\usepackage{tikz}
\usetikzlibrary{arrows}

\usepackage{paralist}
\usepackage{cite}

\title{Fields generated by characters of finite linear groups}

\author[Dawsey]{Madeline Locus Dawsey}
\address{Department of Mathematics, University of Texas at Tyler, Tyler, TX 75799}
\email{mdawsey@uttyler.edu}

\author[Ono]{Ken Ono}
\address{Department of Mathematics, University of Virginia, Charlottesville, VA 22904}
\email{ko5wk@virginia.edu}

\author[Wagner]{Ian Wagner}
\address{Department of Mathematics, Vanderbilt University, Nashville, TN 37212}
\email{ian.c.wagner@vanderbilt.edu}

\begin{document}

\thanks{The first author is grateful for the support of an AMS-Simons Grant. The second author thanks the generous support of the Asa Griggs Candler Fund at Emory University, the Thomas Jefferson Fund at the University of Virginia, and the NSF (DMS-1601306 and DMS-2002265).}

\begin{abstract} 
In previous work \cite{DOW}, the authors confirmed the speculation of  J. G. Thompson that certain multiquadratic fields
are generated by specified character values of sufficiently large alternating groups $A_n$.
 Here we address the natural generalization of this speculation to the finite general linear groups $\GL_m\left(\F_q\right)$ and $\SL_2\left(\F_q\right)$.
\end{abstract}

\maketitle

\section{Introduction and Statement of Results}

Let $G$ be a finite group, and let $K(G)$ denote the field generated over $\Q$ by the character values of $G$.  G. R. Robinson and J. G. Thompson \cite{RT} proved for alternating groups $A_n$ that the $K\left(A_n\right)$ are generally large multiquadratic extensions of $\mathbb{Q}$.  For $n>24$, they proved that
$$K\left(A_n\right)=\Q\left(\{ \sqrt{p^*} \ : \ p\leq n \ {\text {\rm an odd prime with\ } p\neq n-2}\}\right),$$
where $m^*:=(-1)^{\frac{m-1}{2}}m$ for any odd integer $m$.

In a letter \cite{T} to the second author in 1994, Thompson asked whether a refinement of this result
exists that is analogous to the Kronecker--Weber Theorem and the theory of complex multiplication, where
abelian extensions are generated by the values of suitable functions at arguments which determine the ramified primes.  Instead of adjoining special values of analytic functions to number fields, which is the substance of {\it Hilbert's 12th Problem} \cite{H}, Thompson suggested adjoining character values of $A_n$.

The authors recently answered this question \cite{DOW}.  Let $\pi:=\{p_1, p_2,\dots, p_t\}$ be a set of $t\geq 2$ distinct odd primes\footnote{The phenomenon cannot hold when $t=1$ for any $n\not\equiv 0, 1\pmod{p}$.} listed in increasing order. A $\pi$-{\it number} is a positive integer whose prime factors belong to $\pi$.  Thompson's prediction was that in general $K_{\pi}\left(A_n\right)$, the field generated over $\mathbb{Q}$ by the values of $A_n$-characters restricted to elements $\sigma$ of $A_n$ with $\pi$-number order, should be the field $\Q\left(\sqrt{p_1^*}, \sqrt{p_2^*},\dots, \sqrt{p_t^*}\right)$.  Although counterexamples exist for each $\pi$, the authors proved \cite{DOW}, for sufficiently large $n$, that 
$$K_{\pi}\left(A_n\right)=\Q\left(\sqrt{p_1^*}, \sqrt{p_2^*},\dots, \sqrt{p_t^*}\right).$$

In this note we consider Thompson's problem for the finite general linear groups $\GL_m\left(\F_q\right)$ and $\SL_2\left(\F_q\right)$, where $\F_q$ is a finite field.  Throughout, we let $p$ be a prime, and let $q:=p^n$.   Let $\zeta_n$ denote the $n$th root of unity $\zeta_n:=e^{2\pi i/n}$, and let $\mathbb{Z}^\times_{n}:=\left(\mathbb{Z}/n\mathbb{Z}\right)^\times$.  
Thanks to  classical work of Green \cite{G},  it is not difficult to determine the number fields
 generated by the $\GL_m\left(\F_q\right)$-character values. To ease notation, for positive integers $d, r\in \Z^{+}$, we let
 \begin{equation}
 \omega_d(r):=\sum_{k=0}^{d-1}\zeta_{q^d-1}^{rq^k}.
 \end{equation}

\begin{theorem}\label{GLCase}\label{Thm1}
If $m\geq2$, then we have
\begin{equation*}
K\left(\GL_m\left(\F_q\right)\right)=\mathbb{Q}\left(\left\{\omega_d(r)  \ : \ 1\leq d\leq m \ {\rm{and}} \ r\in\mathbb{Z}^\times_{q^d-1}\right\}\right).
\end{equation*}                                                                 
\end{theorem}

Turning to Thompson's speculation, we determine the subfields of $K\left(\GL_m\left(\F_q\right)\right)$ generated by the character values of elements of prime power order $\ell^r$. Confirming Thompson's speculation, we find that these fields systematically correspond to subfields with ramification only at $\ell$.
Let $\ell$ be prime, and let $r\geq1$ be an integer.  Define $K_{\ell^r}\left(\GL_m\left(\F_q\right)\right)$ to be the field generated over $\mathbb{Q}$ by the values of $\GL_m\left(\F_q\right)$-characters evaluated at elements of order $\ell^r$.  For primes $\ell \neq p$, we let $\mathrm{ord}_{\ell^r}(q)$ denote the multiplicative order of $q$ modulo $\ell^r$.  

\begin{theorem}\label{GLEllPower}\label{Thm2}
 Suppose that $\ell \neq p$ is an odd prime. If there is an element of order $\ell^r$ in $\GL_m\left(\F_q\right)$, then $K_{\ell^r}\left(\GL_m\left(\F_q\right)\right)$ is the unique subfield of $\Q\left( \zeta_{\ell^r}\right)$ of degree $\frac{\ell^{r-1}(\ell -1)}{\mathrm{ord}_{\ell^r}(q)}$ over $\Q$.
\end{theorem}

\begin{remark}
If $\tau_{q}\in{\rm{Gal}}\left(\Q\left(\zeta_{\ell^r}\right)/\Q\right)$ satisfies $\tau_q\left(\zeta_{\ell^r}\right)=\zeta_{\ell^r}^{q}$, then  $K_{\ell^r}\left(\GL_{m}\left(\F_q\right)\right) = \Q\left(\zeta_{\ell^r}\right)^{\langle \tau_{q} \rangle}$.
\end{remark}

\begin{remark}
If $\ell=p$, then $K_\ell\left(\mathrm{GL}_m\left(\F_q\right)\right)=\Q$, as shown in Section \ref{rankm}.  The proof of Theorem \ref{Thm2} also shows that $K_{2} \left(\GL_{m} \left(\F_{q} \right) \right) = \Q$ and that
\begin{equation*}
K_{4} \left(\GL_{m} \left(\F_{q} \right) \right) = \begin{cases} \Q(i) & q \equiv 1 \pmod{4} \\ \Q & \text{{\rm otherwise}}. \end{cases}
\end{equation*}
If $\ell^r=2^r$ for $r>2$, then it is not always clear which fields are generated by character values at elements of order $2^r$, since ${\rm{Gal}} \left( \Q\left(\zeta_{2^r}\right)/\Q \right)$ is no longer cyclic.  
\end{remark}

Now we turn to the case of the groups $\SL_2\left(\F_q\right)$.
\begin{theorem}\label{Thm5}
Assuming the notation above, we have that
\begin{equation*}
K\left(\SL_{2}\left(\F_{q}\right)\right) = \begin{cases} \Q\left(\zeta_{q-1} + \zeta_{q-1}^{-1}, \zeta_{q+1} + \zeta_{q+1}^{-1}\right) & \text{if } p=2, \\
\Q\left(\sqrt{q^{*}}, \zeta_{q-1} + \zeta_{q-1}^{-1}, \zeta_{q+1} + \zeta_{q+1}^{-1}\right) & \text{{if }} p>2.
\end{cases}
\end{equation*}
\end{theorem}

Regarding Thompson's speculation for these groups, we obtain the following result.

\begin{theorem}\label{Thm6}
If there is an element of order $\ell^r$ in $\SL_{2}\left(\F_{q}\right)$, then the following are true.
\begin{enumerate}
\item If $p=2$, then we have that
\begin{equation*}
K_{\ell^r}\left(\SL_{2}\left(\F_{q}\right)\right) = \begin{cases} \Q\left(\zeta_{\ell^r} + \zeta_{\ell^r}^{-1}\right) & \text{ if } q \equiv \pm 1 \pmod{\ell^r}, \\
\Q & \text{ if } q \equiv 0 \pmod{\ell^r}.
\end{cases}
\end{equation*}
\item If $p>2$, then we have that
\begin{equation*}
K_{\ell^r}\left(\SL_{2}\left(\F_{q}\right)\right) = \begin{cases} \Q\left(\zeta_{\ell^r} + \zeta_{\ell^r}^{-1}\right) & \text{ if } q \equiv \pm 1 \pmod{\ell^r}, \\
\Q\left(\sqrt{q^{*}}\right) & \text{ if } q \equiv 0 \pmod{\ell^r}.
\end{cases}
\end{equation*}
\end{enumerate}
\end{theorem}

\begin{remark}
 C. Bonnaf\'e \cite{Bonnafe} and T. Shoji \cite{Shoji1, Shoji2}
gave  a method for computing the characters of $\SL_m(\F_q)$ when
$m\geq 3$. We ask whether Theorems~\ref{Thm5} and \ref{Thm6} can be extended
using their work.
\end{remark}

This paper is organized as follows.  In Section \ref{rank2}, we recall some classic facts from the representation theory of $\mathrm{GL}_2\left(\F_q\right)$ and $\mathrm{SL}_2\left(\F_q\right)$.
These facts allow us to deduce Theorems~\ref{Thm5} and \ref{Thm6}, and Theorems~\ref{Thm1} and \ref{Thm2} when $m=2$.  In Section \ref{rankm}, we recall
essential features of the classical work of Green \cite{G} on the representation theory of $\mathrm{GL}_m\left(\F_q\right)$, which we then use to prove Theorems \ref{Thm1} and \ref{Thm2} for ranks $m\geq3$.

\section*{Acknowledgements}
\noindent
The authors thank John Duncan, John G. Thompson and Pham Tiep for helpful conversations.

\section{Rank 2 Finite Linear Groups}\label{rank2}

In Section~\ref{GLrank2} we establish the rank 2 cases of Theorems~\ref{Thm1} and \ref{Thm2}, and we prove Theorems~\ref{Thm5} and \ref{Thm6} in Section \ref{SLrank2}.

\subsection{The $\GL_2\left(\F_q\right)$ cases}\label{GLrank2}

For completeness, we recall the construction of the character tables  of $\GL_{2}\left(\F_q\right)$  following the treatment in \cite{FH}.
It is well-known that $\FF_{q^2}^{\times}$ is isomorphic to a large cyclic subgroup of $\GL_{2}\left(\F_{q}\right)$.
Let $q$ be a power of a prime, and let $\epsilon$ be a generator of $\F_{q}^{\times}$.  We now make explicit the characters of $\GL_{2}(\F_q)$.  When $q$ is odd, $1$ and $\sqrt{\epsilon}$ form a basis for $\F_{q^2}^{\times}$ as a vector space over $\F_q$, and so an element $\zeta \in \F_{q^2}^{\times}$ is of the form $\zeta = x + \sqrt{\epsilon}\,y$.  For an element $\zeta \in \F_{q^2}^{\times}$ with $q$ odd, define the matrix $d_{\zeta}$ by
\begin{equation*}
d_{\zeta} := \begin{pmatrix} x & \epsilon y \\ y & x \end{pmatrix}.
\end{equation*}
We have the isomorphism $\{d_{\zeta} \} \cong \F_{q^2}^{\times}$ where we have made the identification $d_{\zeta} \leftrightarrow \zeta$.  When $q$ is even, we similarly identify $\zeta \in \F_{q^2}^{\times}$ with the matrix $d_{\zeta} \in \GL_{2}(\F_q)$ so that the order of $d_{\zeta}$ in $\GL_{2}(\F_q)$ is equal to the order of $\zeta$ in $\F_{q^2}^{\times}$.
Using the notation above, the conjugacy classes of $\GL_{2}\left(\F_{q}\right)$ are given in the table below.

\smallskip
\begin{center}
  \begin{tabular}{|c|c|c|}
    \hline
     Representative & $\#$ elements & $\#$ classes  \\ \hline
    $a_{x} = \left(\begin{smallmatrix} x & 0 \\ 0 & x \end{smallmatrix}\right)$ & $1$ & $q-1$ \\ \hline
    $b_{x} = \left(\begin{smallmatrix} x & 1 \\ 0 & x \end{smallmatrix}\right)$ & $q^2 -1$ & $q-1$  \\ \hline
    $c_{x,y} =\left(\begin{smallmatrix} x & 0 \\ 0 & y \end{smallmatrix}\right),$ \, $x \neq y$ & $q^2 +q$ & $\dfrac{(q-1)(q-2)}{2}$  \\ \hline
    $d_{\zeta}, \, \zeta \in \F_{q^{2}} \setminus \F_{q}$ & $q^2 -q$ & $\dfrac{q(q-1)}{2}$  \\ \hline
  \end{tabular}
  
  \hskip.15in
  
{\text {\rm Table 1. Conjugacy classes of  $\GL_{2}\left(\F_{q}\right)$}}
\end{center}
\smallskip

We now turn to the construction of the character table.
The group $\FF_{q}^{\times} = \{1, a, \dots, a^{q-2} \}$ has $q-1$ characters, say
$\alpha_{k}: \FF_{q}^{\times} \to \C^{\times},$ defined by
$\alpha_k(a):= \zeta_{q-1}^{k-1}$
for $1\leq k\leq q-1$.  Similarly, the multiplicative group $\FF_{q^2}^\times$ has $q^2-1$ characters which we denote by $\phi_{n}$, for $1 \leq n \leq q^2 -1$.  

We employ these characters to describe $\mathrm{Irr}\left(\GL_2\left(\F_q\right)\right)$, the $q^2 -1$ irreducible representations of $\GL_{2}\left(\FF_q\right)$.  The permutation representation of $\GL_{2}\left(\FF_q\right)$ on $\mathbb{P}^{1}\left(\FF_q\right)$ has dimension $q+1$ and contains the trivial representation.  The complementary $q$-dimensional representation is irreducible and will be denoted by $V$.  There are also $q-1$ irreducible one-dimensional representations $U_{k}$ given by $U_{k}(g) = \alpha_{k}(\det(g))$.  Note that $\det\left(d_{\zeta}\right)= N_{\FF_{q^2}/\FF_{q} }(\zeta) = \zeta^{q+1}$.  The $q-1$ representations given by $V_{k} = V \otimes U_{k}$ are also irreducible.  For each pair $\alpha_j,\alpha_k$ of characters of $\FF_{q}^{\times}$, there is a character $\psi$ of the subgroup $B$,  the Borel subgroup consisting of upper triangular matrices,
such that
$$\psi:\begin{pmatrix} a & b \\ 0 & d \end{pmatrix}\longmapsto\alpha_{j}(a) \alpha_{k}(d).$$
Let $W_{j,k}$ be the representation of $\GL_{2}\left(\FF_{q}\right)$ induced from the representation of $B$ associated to $\psi$.  Note that $W_{j,j} \cong U_{j} \oplus V_{j}$ and $W_{j,k} \cong W_{k,j}$, and that there are $(q-1)(q-2)/2$ representations $W_{j,k}$ of dimension $q+1$ for $j\neq k$.  It can be shown that the final $q(q-1)/2$ irreducible representations arise from first inducing the characters $\phi$ of $\FF_{q^2}^{\times}$ with $\phi \neq \phi^{q}$ and $\phi |_{\FF_{q}^{\times}}=\alpha_{k}$ for some $k,$ and then recognizing that ${\rm Ind}(\phi)$ and $V_{1} \otimes W_{j,1}$ contain many of the same representations.  It can quickly be checked that $\chi_{\phi} = \chi_{V_{1} \otimes W_{j,1}} - \chi_{W_{j, 1}} - \chi_{{\rm Ind}(\phi)}$ is the character of an irreducible representation of $V_{1} \otimes W_{j,1}$ which we denote by $X_{\phi}$.  The values of this character are given in Table 2 below. To summarize, we have the following table.

\smallskip
\begin{center}
\resizebox{17cm}{!}{
  \begin{tabular}{|c|c c c c|}
    \hline
    $\GL_{2}\left(\FF_{q}\right)$ & $a_{x} = \left(\begin{matrix} x & 0 \\ 0 & x \end{matrix}\right)$ & $b_{x} = \left(\begin{matrix} x & 1 \\ 0 & x \end{matrix}\right)$ & $c_{x,y} = \left(\begin{matrix} x & 0 \\ 0 & y \end{matrix}\right)$, \, $x\neq y$ & $d_{\zeta}$, \, $\zeta\in\F_{q^2}\setminus\F_q$ \\ \hline
    $U_{k}$ & $\alpha_{k}\left(x^2\right)$ & $\alpha_{k}\left(x^2\right)$ & $\alpha_{k}(xy)$ & $\alpha_{k}\left(\zeta^{q+1}\right)$ \\
    $V_{k}$ & $q\alpha_{k}\left(x^2\right)$ & $0$ & $\alpha_{k}(xy)$ & $-\alpha_{k}\left(\zeta^{q+1}\right)$ \\
    $W_{j,k}$ & $(q+1)\alpha_{j}(x) \alpha_{k}(x)$ & $\alpha_{j}(x) \alpha_{k}(x)$ & $\alpha_{j}(x) \alpha_{k}(y) + \alpha_{j}(y) \alpha_{k}(x)$ & $0$ \\
    $X_{\phi}$ & $(q-1)\phi(x)$ & $-\phi(x)$ & $0$ & $-\left(\phi(\zeta) + \phi\left(\zeta^{q}\right)\right)$ \\ \hline
  \end{tabular}
  }
  
\hskip.15in
  
{\text {\rm Table 2. Character table of  $\GL_{2}\left(\F_{q}\right)$}}
 \smallskip
\end{center}

 We now describe the character values of $\GL_2\left(\F_q\right)$ based on the group orders.  To this end, we first determine these orders.  
 It is clear that the order of $a_x$ in $\GL_2\left(\F_q\right)$ is equal to the order of $x$ in $\FF_{q}^{\times}$.  Therefore, for each $d\mid(q-1)$, there are $\varphi(d)$ conjugacy classes $\left[a_x\right]$ consisting of elements of order $d$, where $\varphi$ denotes Euler's totient function.
   Notice that $$b_{x}^{n} = \begin{pmatrix} x^n & n x^{n-1} \\ 0 & x^n \end{pmatrix},$$ so in order for $b_{x}^{n}$ to be the identity matrix, we must have that the order of $x$ divides $n$ and $n \equiv 0 \pmod{p}$.  Since the order of $x$ is coprime to $p$, we have that for each $d\mid(q-1)$, there are $\varphi(d)$ conjugacy classes $\left[b_x\right]$ consisting of elements of order $pd$.  
   The order of $c_{x,y}$ is the least common multiple of the order of $x$ and the order of $y$.  Therefore, for each $d>1$ with $d\mid(q-1)$, there are $\varphi(d) \left( d - \frac{\varphi(d) + 1}{2} \right)$ conjugacy classes $\left[c_{x,y}\right]$ consisting of elements of order $d$.  
   Finally, for each $d\mid\left(q^2 -1\right)$ with $d\nmid (q-1)$, there are $\varphi(d)/2$ conjugacy classes $\left[d_\zeta\right]$ consisting of elements of order $d$.

Define $K_{d}(G):= \Q \left(\{ \chi(g): g \in G \text{ has order } d \text{ and } \chi \in \text{Irr}(G)\}\right)$.  We have the following lemma relating character values at elements of order $d$ to those at elements of order $pd$.
\begin{lemma}\label{L1}
If $d\mid(q-1)$, then
$
K_{d}\left(\GL_2\left(\F_{q}\right)\right) = K_{pd}\left(\GL_{2}\left(\F_{q}\right)\right) = \Q\left(\zeta_{d}\right).
$
\end{lemma}
\begin{proof}
The conjugacy classes $\left[a_{x}\right]$ and $\left[c_{x,y}\right]$ may contain elements of order $d\mid(q-1)$.  It suffices to restrict our attention to the values of the character $U_k$ at elements in the conjugacy classes $c_{x,1}$, where $x$ has order $d$.  The value of $U_2$ is $\alpha_{2}(xy) = \alpha_{2}(x) = \zeta_{q-1}^{t}$ for some $0\leq t\leq q-2$.  Since the order of $x$ is $d$, we must have that $\zeta_{q-1}^{t} = \zeta_{d}^{r}$ for some $r$ coprime to $d$.  In fact, cycling through all of the characters $\alpha_{k}$ shows that every $d$th root of unity actually appears as a character value for $d\mid(q-1)$.  

We now turn to the conjugacy classes $\left[b_x\right]$ with elements of order $pd$.  For such elements $b_x$, the corresponding element $x$ of $\F_q^\times$ must have order $d$.  Therefore, the desired $d$th roots of unity arise as the values of the characters $W_{j,1}$, $1\leq j\leq q-1$, and the values of $X_{\phi}$.
\end{proof}

\begin{lemma}\label{L2}
Suppose that $d\mid\left(q^2 -1\right)$ but $d \nmid (q-1)$.  Then
\begin{equation*}
K_{d}\left(\GL_{2}\left(\FF_{q}\right)\right) = \Q\left(\left\{ \zeta_{d}^{r} + \zeta_{d}^{qr}: 1 \leq r \leq d \right\}\right).
\end{equation*}
\end{lemma}
\begin{proof}
The only conjugacy classes containing elements of order $d\mid\left(q^2-1\right)$ with $d\nmid(q-1)$ are the classes $\left[d_{\zeta}\right]$, where $\zeta=x+\sqrt{\epsilon}\,y$ has order $d$ in $\F_{q^2}^\times$.  It is then clear that the desired sums $\zeta_d^r+\zeta_d^{qr}$ arise from the values of the characters $X_{\phi}$ at these elements.
\end{proof}

It is natural to ask which fields are generated by the sums of roots of unity appearing in Lemma \ref{L2}.  Here we make some initial observations for small values of $d$.  In order for the conditions of Lemma \ref{L2} to hold, we must have that $q^2 \equiv 1 \pmod{d}$ and $q \not\equiv 1 \pmod{d}$.  For some small values of $d$, the only solutions to this pair of congruences are $q \equiv -1 \pmod{d}$.  In these special cases, we can identify the fields in question.  
For example, if $d=3,4,6$ and $q\equiv-1\pmod{d}$, then $\zeta_{d}^{r} + \zeta_{d}^{-r} \in \Z$, and so $K_{d}\left(\GL_{2}\left(\FF_{q}\right)\right) = \Q$.  
If $d=5$ and $q \equiv -1 \pmod{d}$, then $\zeta_{5}^{r} + \zeta_{5}^{-r} = \frac{1}{2} \left( -1+ \left(\frac{r}{5} \right) \sqrt{5} \right)$, and so $K_{5}\left(\GL_{2}\left(\FF_{q}\right)\right) = \Q(\sqrt{5})$.  
If $d=8$ and $q\equiv-1\pmod{d}$, then a similar argument implies that
\begin{equation*}
K_{8}\left(\GL_2\left(\F_{q}\right)\right) = \begin{cases} \Q\left(\zeta_{8}\right) & \text{ if } q \equiv 1 \pmod{8}, \\
\Q\left(\sqrt{-2}\right) & \text{ if } q \equiv 3 \pmod{8}, \\
\Q(i) & \text{ if } q \equiv 5 \pmod{8}, \\
\Q\left(\sqrt{2}\right) & \text{ if } q \equiv 7 \pmod{8}.
\end{cases}
\end{equation*}

For general orders $d$, Lemmas \ref{L1} and \ref{L2} imply the following two theorems, which are restatements of Theorems \ref{Thm1} and \ref{Thm2} for $m=2$.
\begin{theorem}\label{Thm3}
We have that
\begin{equation*}
K\left(\GL_{2}\left(\F_{q}\right)\right) = \Q \left(\zeta_{q-1}, \left\{\zeta_{q^2 -1}^{r} + \zeta_{q^2 -1}^{qr} \ : \ 1 \leq r \leq \frac{q^2 -1}{2}\right\} \right).
\end{equation*}
\end{theorem}

The previous theorem, and the following theorem addressing Thompson's speculation, prove the rank 2 cases
of Theorems~\ref{Thm1} and \ref{Thm2}.

\begin{theorem}\label{Thm4} Suppose that $\ell$  is an odd prime and that
there is an element of order $\ell^r$ in $\GL_{2}\left(\F_{q}\right)$. Then we have that
\begin{equation*}
K_{\ell^r}\left(\GL_{2}\left(\FF_{q}\right)\right) = \begin{cases} 
\Q\left(\zeta_{\ell^r}\right) & \text{if } q \equiv 1 \pmod{\ell^r}, \\
\Q\left(\zeta_{\ell^r} + \zeta_{\ell^r}^{-1}\right) & \text{if } q \equiv -1 \pmod{\ell^r}, \\
\Q & \text{if } q \equiv 0 \pmod{\ell^r}.
\end{cases}
\end{equation*}
\end{theorem}

\begin{rmk}
The number field $\Q\left(\zeta_{\ell^r} + \zeta_{\ell^r}^{-1}\right)$ is the maximal totally real subfield of the cyclotomic field $\Q\left(\zeta_{\ell^r}\right)$ and has degree $\frac{\ell^{r-1}(\ell -1)}{2}$ over $\Q$.
\end{rmk}

The case when $\ell=2$ is more complicated than the cases described above.  In this case, it is possible that $q \equiv 2t+1 \pmod{2^r}$ for any $0 \leq t \leq 2^{r-1} -1$.  However, there are still a few congruence classes of $q$ that yield simple fields.  For example, if $q \equiv \pm1 \pmod{2^{r}}$, then the conclusion of Theorem \ref{Thm4} holds.  If $q \equiv 2^{r-1} \pm 1 \pmod{2^r}$, then we have that
\begin{equation*}
K_{2^r}\left(\GL_2\left(\F_q\right)\right)=\begin{cases} \Q\left(\zeta_{2^{r-1}} + \zeta_{2^{r-1}}^{-1}, \zeta_{2^{r}} - \zeta_{2^r}^{-1}\right) & \text{if }q \equiv 2^{r-1} -1 \pmod{2^r}, \\
\Q\left(\zeta_{2^{r-1}}\right) & \text{if }q \equiv 2^{r-1} +1 \pmod{2^r}.
\end{cases}
\end{equation*}

\subsection{Proof of Theorems~\ref{Thm5} and \ref{Thm6}}\label{SLrank2}

As in the previous subsection, the proofs follow from the explicit  description of the character tables of $\SL_2\left(\F_q\right)$.
We consider the cases where $q$ is odd (resp. even) separately.

\subsubsection{The case when $q$ is odd}\label{ss1}
When $q$ is odd, the conjugacy class $\left[b_{1}\right]$ in $\GL_2\left(\F_q\right)$ splits into two conjugacy classes $\left[b_{1,1}\right]$ and $\left[b_{1, \epsilon}\right]$ in $\SL_2\left(\F_q\right)$, and the conjugacy class $\left[b_{-1}\right]$ splits into $\left[b_{-1,1}\right]$ and $\left[b_{-1, \epsilon}\right]$.  These split conjugacy classes are given by the following table.

\smallskip
\begin{center}
  \begin{tabular}{|c|c|c|}
    \hline
     Representative & $\#$ elements & $\#$ classes  \\ \hline
    $\left(\begin{smallmatrix} 1 & 0 \\ 0 & 1 \end{smallmatrix}\right)$ & $1$ & $1$ \\ \hline
    $\left(\begin{smallmatrix} -1 & 0 \\ 0 & -1 \end{smallmatrix}\right)$ & $1$ & $1$ \\ \hline
    $b_{1,1}=\left(\begin{smallmatrix} 1 & 1 \\ 0 & 1 \end{smallmatrix}\right)$ & $\dfrac{q^2 -1}{2}$ & $1$  \\ \hline
    $b_{1, \epsilon} =\left(\begin{smallmatrix} 1 & \epsilon \\ 0 & 1 \end{smallmatrix}\right)$ & $\dfrac{q^2 -1}{2}$ & $1$  \\ \hline
    $b_{-1,1}=\left(\begin{smallmatrix} -1 & 1 \\ 0 & -1 \end{smallmatrix}\right)$ & $\dfrac{q^2 -1}{2}$ & $1$  \\ \hline
    $b_{-1, \epsilon} =\left(\begin{smallmatrix} -1 & \epsilon \\ 0 & -1 \end{smallmatrix}\right)$ & $\dfrac{q^2 -1}{2}$ & $1$  \\ \hline
    $c_{x} = \left(\begin{smallmatrix} x & 0 \\ 0 & x^{-1} \end{smallmatrix}\right), \, x \neq \pm 1$ & $q^2 +q$ & $\dfrac{q-3}{2}$  \\ \hline
    $d_{\zeta}= \left(\begin{smallmatrix} x & \epsilon y \\ y & x \end{smallmatrix}\right), \, x \neq \pm 1, \, y \neq 0, \, \zeta^{q+1}=1$ & $q^2 -q$ & $\dfrac{q-1}{2}$  \\ \hline
  \end{tabular}
  
  \hskip.15in
  
{\text {\rm Table 3. Split conjugacy classes of  $\SL_{2}\left(\F_{q}\right)$} for odd $q$}
 \smallskip
\end{center}

Define the set $C := \left\{ \zeta \in \F_{q^2}^{\times}: \zeta^{q+1}=1 \right\}$.  
Here we give the restrictions of the representations of $\GL_{2}\left(\F_{q}\right)$ to $\SL_2\left(\F_q\right)$.  All of the representations $U_{k}$ restrict to the trivial representation $U$.  
The restriction $V$ of $V_{k}$ is also an irreducible representation.  
The restriction $W_{j}$ of $W_{j,1}$ is irreducible if $\alpha_{j}^{2} \neq 1$.  
It is also clear that $W_{j} \cong W_{k}$ if $j= \pm k$, and so there are $(q-3)/2$ irreducible representations $W_j$ of dimension $q+1$.  
If $\tau$ is the nontrivial quadratic character of $\F_{q}^{\times}$, then $W_{\tau} = W' \oplus W''$, where both $W'$ and $W''$ are irreducible representations.  
The restriction of $X_{\phi}$ is determined by the restriction of $\phi$ to $C$.  
The restrictions of $\phi$ and $\phi^{-1}$ to $C$ yield the same character, so that if $\phi^2 \neq 1$, then there are $(q-1)/2$ irreducible representations of dimension $q-1$.  
If $\psi$ is the nontrivial quadratic character of $C$, then $X_{\psi} = X' \oplus X''$, where both $X'$ and $X''$ are irreducible.  
To complete the character table of $\SL_2\left(\F_q\right)$, it remains to identify the irreducible representations $W', W'', X'$ and $X''$.  
By studying the index two subgroup $H \subset \GL_{2}\left(\F_{q}\right)$ of matrices with square determinant, one can show that $W'$ and $W''$ are conjugate representations of dimension $(q+1)/2$, and $X'$ and $X''$ are conjugate representations of dimension $(q-1)/2$.  
On non-split conjugacy classes, the character values of $W',W''$ and $X',X''$ are half the character values of $W_j$ and $X_{\psi}$, respectively.  
On the split conjugacy classes, the character values can be determined from character relations.

The character table of $\SL_2\left(\F_q\right)$ in the case when $q$ is odd is shown below.  We note that $\tau(-1) = \psi(-1) = (-1)^{(q-1)/2}$ and $\tau(\epsilon) = -1$.

\smallskip
\begin{center}
\resizebox{17cm}{!}{
  \begin{tabular}{|c|c c c c|}
    \hline
     $\SL_{2}\left(\FF_{q}\right)$ & $2$ & $\dfrac{q^2 -1}{2}$ & $q^2 +q$ & $q^2 -q$ \\
   $q$ odd & $\begin{pmatrix} \pm1 & 0 \\ 0 & \pm1 \end{pmatrix}$ & $b_{x,y} = \begin{pmatrix} x & y \\ 0 & x \end{pmatrix}, \, x \in \{\pm 1\}, \, y \in \{1, \epsilon \}$ & $c_{x} = \begin{pmatrix} x & 0 \\ 0 & x^{-1} \end{pmatrix}, \, x \neq \pm 1$ & $d_{\zeta} = \begin{pmatrix} x & \epsilon y \\ y & x \end{pmatrix}, \, x \neq \pm 1, \, y \neq 0, \, \zeta^{q+1}=1$ \\ \hline
    $U$ & $1$ & $1$ & $1$ & $1$ \\
    $V$ & $q$ & $0$ & $1$ & $-1$ \\
    $W_{k}, \, \alpha_{k}^2 \neq 1$ & $(q+1)\alpha_{}(\pm 1)$ & $\alpha_{k}(x)$ & $\alpha_{k}(x) +\alpha_{k}\left(x^{-1}\right)$ & $0$ \\
    $X_{\phi}, \, \phi^2 \neq 1$ & $(q-1)\phi(\pm1)$ & $-\phi(x)$ & $0$ & $-\left(\phi(\zeta) + \phi\left(\zeta^{-1}\right)\right)$ \\
    $W'$ & $[(q+1)/2] \, \tau(\pm1)$ & $[\tau(x)/2] \left( 1 + \tau(y) \sqrt{q^{*}} \right)$ & $\tau(x)$ & $0$ \\
    $W''$ & $[(q+1)/2] \, \tau(\pm 1)$ & $[\tau(x)/2] \left( 1 - \tau(y) \sqrt{q^{*}} \right)$ & $\tau(x)$ & $0$ \\
    $X'$ & $[(q-1)/2] \, \psi(\pm1)$ & $[\tau(x)/2] \left( -1 + \tau(y) \sqrt{q^{*}} \right)$ & $0$ & $-\psi(y)$ \\
    $X''$ & $[(q-1)/2] \, \psi(\pm1)$ & $[\tau(x)/2] \left( -1 - \tau(y) \sqrt{q^{*}} \right)$ & $0$ & $-\psi(\zeta)$ \\ \hline
  \end{tabular}
  }
  
  \hskip.15in
  
{\text {\rm Table 4. Character table of  $\SL_{2}\left(\F_{q}\right)$ for odd $q$}}
 \smallskip
  
\end{center}

We see that $\pm I_2$ has order $1$ or $2$, and all the character values of $\pm I_2$ are integers.  The matrix $b_{x,y}$ has order $p$ when $x=1$ and order $2p$ when $x=-1$, and some of the character values of $b_{x,y}$ include a factor of $\sqrt{q^{*}}$.  For each $d\mid(q-1)$, there are matrices $c_{x}$ and $d_\zeta$ of order $d$ which yield character values of the form $\zeta_{d}^{r} + \zeta_{d}^{-r}$. This completes the proof of Theorems~\ref{Thm5} and \ref{Thm6} when $q$ is odd.

\subsubsection{The case when $q$ is even}\label{ss2}
When $q$ is even, the conjugacy classes do not split as they do when $q$ is odd.  The conjugacy classes are given in the following table.
 
 \smallskip
 \begin{center}
  \begin{tabular}{|c|c|c|}
    \hline
     Representative & $\#$ elements & $\#$ classes  \\ \hline
    $\left(\begin{smallmatrix} 1 & 0 \\ 0 & 1 \end{smallmatrix}\right)$ & $1$ & $1$ \\ \hline
    $\left(\begin{smallmatrix} 1 & 1 \\ 0 & 1 \end{smallmatrix}\right)$ & $q^2 -1$ & $1$  \\ \hline
    $c_{x} = \left(\begin{smallmatrix} x & 0 \\ 0 & x^{-1} \end{smallmatrix}\right), \, x \neq \pm 1$ & $q^2 +q$ & $\dfrac{q-2}{2}$  \\ \hline
    $d_{\zeta}, \, \zeta^{q+1}=1 \neq \zeta$ & $q^2 -q$ & $\dfrac{q}{2}$ \\ \hline
  \end{tabular}

\hskip.15in
  
{\text {\rm Table 5. Conjugacy classes of  $\SL_{2}\left(\F_{q}\right)$} for even $q$}
\end{center}
 \smallskip

It is straightforward to work out the representations of $\mathrm{SL}_2\left(\mathbb{F}_q\right)$ in this case using the same method as \cite{FH}, by restricting the representations of $\mathrm{GL}_2\left(\mathbb{F}_q\right)$.  The restrictions of the representations are similar to those in the case when $q$ is odd, except that $W_{j}$ and $X_{\phi}$ no longer split at the nontrivial quadratic characters.  We have the following character table.

\vspace{-.4cm}

\smallskip

\begin{center}
\resizebox{17cm}{!}{
  \begin{tabular}{|c|c c c c|}
    \hline
     $\SL_{2}\left(\FF_{q}\right)$ & $1$ & $q^2 -1$ & $q^2 +q$ & $q^2 -q$ \\
   $q$ even & $\begin{pmatrix} 1 & 0 \\ 0 & 1 \end{pmatrix}$ & $\begin{pmatrix} 1 & 1 \\ 0 & 1 \end{pmatrix}$ & $c_{x} = \begin{pmatrix} x & 0 \\ 0 & x^{-1} \end{pmatrix}, \, x \neq \pm 1$ & $d_{\zeta}, \, \zeta^{q+1}=1 \neq \zeta$ \\ \hline
    $U$ & $1$ & $1$ & $1$ & $1$ \\
    $V$ & $q$ & $0$ & $1$ & $-1$ \\
    $W_{j}$ & $q+1$ & $1$ & $\alpha_{j}(x) +\alpha_{j}\left(x^{-1}\right)$ & $0$ \\
    $X_{\phi}$ & $q-1$ & $-1$ & $0$ & $-\left(\phi(\zeta) + \phi\left(\zeta^{-1}\right)\right)$ \\ \hline
  \end{tabular}
  }
  
  \hskip.15in
  
{\text {\rm Table 6. Character table of  $\SL_{2}\left(\F_{q}\right)$ for even $q$}}
 \smallskip
\end{center}

\smallskip

The order of $I_2$ is 1, and the order of the element in the second column of the character table is $2$.  Both of these elements yield integer character values.  For each $d\mid(q-1)$, there are matrices $c_{x}$ and $d_\zeta$ of order $d$ which yield character values of the form $\zeta_{d}^{r} + \zeta_{d}^{-r}$.
This completes the proof of Theorems~\ref{Thm5} and \ref{Thm6} when $q$ is even.

\section{$\GL_m\left(\F_q\right)$ for Dimensions $m\geq 3$}\label{rankm}

Here we define the notation required to describe the representation theory of $\mathrm{GL}_m\left(\F_q\right)$ for general rank $m\geq3$, and we prove Theorems~\ref{Thm1} and \ref{Thm2} for ranks $m\geq 3$. 

\subsection{Proof of Theorem~\ref{Thm1}}
We first describe the conjugacy classes of $\mathrm{GL}_m\left(\F_q\right)$.  Let $f(t)$ be the monic polynomial $$f(t) = \sum_{i=0}^{d} a_{i} t^{i} \in \F_{q}[t]$$ with $a_{d}=1$, and let $U_{1}(f)$ be the companion matrix of $f$.  For any positive integer $n$, define the Jordan block matrix
\begin{equation*}
U_{n}(f) := \begin{pmatrix} U_{1}(f) & I_{d} & 0 & \cdots & 0 \\
0 & U_{1}(f) & I_{d} & \cdots & 0 \\
\vdots & \vdots & \ddots & \ddots & \vdots \\
\vdots & \vdots & \vdots & \ddots & I_{d} \\
0 & 0 & 0 & \cdots & U_{1}(f) \end{pmatrix}.
\end{equation*}
For a partition $\lambda = (\lambda_{1}, \lambda_{2}, \dots, \lambda_{k})$ of $m$, let $U_\lambda(f)$ be the matrix defined by
\begin{equation*}
U_{\lambda}(f) := {\rm{diag}} \left \{U_{\lambda_{1}}(f), U_{\lambda_{2}}(f), \dots, U_{\lambda_{k}}(f) \right \} = \bigoplus_{i=1}^{k} U_{\lambda_{i}}(f).  
\end{equation*}
Let $A \in \GL_{m}\left(\F_q\right)$ with characteristic polynomial
\begin{equation*}
f_{A} = f_{1}^{k_{1}} f_{2}^{k_{2}} \cdots f_{N}^{k_{N}},
\end{equation*}
where $f_{i} \neq t$ are irreducible polynomials over $\F_q$.  Then $A$ is conjugate to its Jordan canonical form
\begin{equation*}
{\rm{diag}} \left \{ U_{\lambda_{1}}(f_{1}), U_{\lambda_{2}}(f_{2}), \dots, U_{\lambda_{N}}(f_{N}) \right \} =\bigoplus_{i=1}^{N} U_{\lambda_{i}}(f_{i})
\end{equation*}
where $\lambda_{i}$ is some partition of $k_{i}$.

Conjugacy classes in $\GL_{m}\left(\F_q\right)$ are defined by $$\left(\left\{f_{i} \right\}, \left\{d_{i} \right\}, \left\{k_{i} \right\}, \left\{\lambda_{i} \right\}\right)_{i=1}^{N},$$ where $f_{i} \neq t$ is an irreducible polynomial over $\F_{q}$ of degree $d_{i}$, $\lambda_{i}$ is a partition of $k_{i}$, $N$ is called the length of the conjugacy class, and $$\sum_{i=0}^{N} k_{i} d_{i} = m.$$  We call a conjugacy class primary if $f_{i}=1$ for all $i$ except one.  We say that two conjugacy classes 
\begin{equation*}
\left(\left\{f_{i} \right\}, \left\{d_{i} \right\}, \left\{k_{i} \right\}, \left\{\lambda_{i} \right\}\right)_{i=1}^{N} \quad \text{and} \quad \left(\left\{g_{i} \right\}, \left\{e_{i} \right\}, \left\{w_{i} \right\}, \left\{\nu_{i} \right\}\right)_{i=1}^{N'}
\end{equation*}
are of the same type if $N=N'$ and there is a permutation $\sigma \in S_{m}$ such that $e_{\sigma(i)} = d_{i}$, $w_{\sigma(i)} = k_{i}$, and $\nu_{\sigma(i)} = \lambda_{i}$, but the $f_{i}$ and $g_{i}$ are allowed to differ.  When $f_i$ and $g_i$ have different roots, they yield two different conjugacy classes of the same type.

In \cite{G}, Green shows that the values of an irreducible character of $\GL_{m}\left(\F_q\right)$ on classes of the same type can be described by a single formula.  Since $f_{i}$ is irreducible of degree $d_{i}$, the roots of $f_{i}$ are of the form $\alpha_{i}, \alpha_{i}^{q}, \dots, \alpha_{i}^{q^{d_{i}-1}}$ for some $\alpha_{i} \in \F_{q^{d_{i}}}^{\times}$, where $\alpha_i$ is not in $F_{q^{\delta}}^{\times}$ for any divisor $\delta$ of $d_{i}$.  Elements in the conjugacy class $\left(\left\{f_{i} \right\}, \left\{d_{i} \right\}, \left\{k_{i} \right\}, \left\{\lambda_{i} \right\}\right)_{i=1}^{N}$ have order given by ${\rm{lcm}} \left( \mathrm{ord}\left(\alpha_{i}\right) \right)$, where $\mathrm{ord}\left(\alpha_i\right)$ is the order of $\alpha_{i}$ in $\F_{q^{d_{i}}}^{\times}$.

The values of the cuspidal characters are straightforward to describe.  Let $\theta$ be a character of $\F_{q^m}^{\times}$, and denote its conjugate characters by $\theta_{i} = \theta^{q^i}$ for each $0 \leq i \leq m-1$.  We say that $\theta$ is non-decomposable if all of its conjugates are distinct.  For an irreducible polynomial $f(t) \in \F_{q}[t]$ of degree $d = {\rm{deg}}(f)$ with roots $\alpha, \alpha^{q}, \dots, \alpha^{q^{d-1}}$, define
\begin{equation*}
\theta(f) = \sum_{k=0}^{d-1} \theta \left(\alpha^{q^k} \right).
\end{equation*}
If $d\mid m$, then we define the cuspidal character $\chi_\theta$ to be the class function
\begin{equation*}
\chi_{\theta}(g) = \begin{cases}
\phi_{f}(q) \theta(f) & \text{if } [g] \ \text{is the primary class associated to} \ f, \\
0 & \text{otherwise},
\end{cases}
\end{equation*}
where $\phi_{f}(q) \in \Z$.  If $\theta$ is a  non-decomposable character of $\F_{q^{m}}^{\times}$, then $(-1)^{m-1} \chi_{\theta}$ is an irreducible character of $\GL_{m}\left(\F_{q}\right)$.  In fact, Theorem 13 of \cite{G} shows that $\chi$ is an irreducible character of $\GL_{m}\left(\F_q\right)$ if and only if it is a product of (possibly repeated) cuspidal characters such that the sum of the degrees of the associated polynomials is $m$.  This proves Theorem \ref{Thm1}.

\subsection{Proof of Theorem~\ref{Thm2}}
Recall that
\begin{equation*}
\left|\GL_{m}\left(\F_q\right)\right| = q^{\frac{m(m-1)}{2}} \prod_{k=1}^{m} \Phi_{k}(q)^{\lfloor m/k \rfloor},
\end{equation*}
where $\Phi_{k}$ is the $k$th cyclotomic polynomial.  Therefore, if there is an element of order $\ell^r$ in $\GL_{m}\left(\F_q\right)$, then either $\ell^r=p$ or $\ell^r\mid \Phi_{k}(q)$ for some $k \leq m$.  It is clear from the description of the characters above that in the first case, if $\ell^r=p$, then all of the character values are integers.  In the second case, if $k_{r} = {\rm{ord}}_{\ell^r}(q)$ is multiplicative order of $q$ modulo $\ell^r$, then $\ell^r \mid \Phi_{k_r}(q)$ and $\ell^r\nmid\Phi_{k_s}(q)$ for all $s<r$.  Therefore, an element of order exactly $\ell^r$ must come from a conjugacy class associated to an irreducible polynomial of degree $k_r$.

The order of an element in $\GL_{m}\left(\F_q\right)$ is the least common multiple of the orders of the roots of the polynomials associated to its conjugacy class.  For an element of order $\ell^r$, these orders must be powers $\ell^a$ with $a\leq r$.  Therefore, elements of order $\ell^r$ must come from conjugacy classes associated to an irreducible degree $k_r$ polynomial and possibly also irreducible degree $k_{a}$ polynomials with $a<r$.  The character values at these elements are integer multiples of
\begin{align*}
&\prod_{a=1}^{r} \left( \sum_{t=0}^{k_{a}-1} \theta_{a} \left(\alpha_{a}^{q^{t}} \right) \right) = \prod_{a=1}^{r} \left(\sum_{t=0}^{k_{a}-1} \zeta_{\ell^a}^{i_{a} q^{t}} \right),
\end{align*}
where $i_{r} \in \Z_{\ell^r}^\times$ and $i_{a} \in \Z_{\ell^a}$ for $a<r$.

We define the following notation for the lemma below.  Let $G_{\ell^r} := {\rm{Gal}}\left(\Q\left(\zeta_{\ell^r}\right)/ \Q\right)$, so that $G_{\ell^r}\cong\Z_{\ell^r}^\times$.  Define $\tau_{q} \in G_{\ell^r}$ to be the automorphism such that $\tau_{q}(\zeta_{\ell^r}) = \zeta_{\ell^r}^{q}$.
Theorem~\ref{Thm2} follows as a consequence of the following lemma.

\begin{lemma}\label{Lemma3.1}
We have that
\begin{equation*}
\Q \left(\prod_{a=1}^{r} \left(\sum_{t=0}^{k_{a}-1} \zeta_{\ell^a}^{i_{a} q^{t}} \right) \right) = \Q(\zeta_{\ell^r})^{\langle \tau_{q} \rangle},
\end{equation*}
with $i_{r} \in \Z_{\ell^r}^\times$, is the unique subfield of $\Q(\zeta_{\ell^r})$ of degree $\frac{\ell^{r-1}(\ell -1)}{k_r}$ over $\Q$.
\end{lemma} 

\begin{proof}
For simplicity, we take $i_{r}=1$.  If $\tau \in G_{\ell^r}$, then $\tau(\zeta_{\ell^r}) = \zeta_{\ell^r}^{s}$ for some $s \in \Z_{\ell^r}^\times$, and so we have that
\begin{equation} \label{C}
\tau \left( \prod_{a=1}^{r} \left( \sum_{t=0}^{k_{a} -1} \zeta_{\ell^a}^{i_{a} q^{t}} \right) \right)= \prod_{a=1}^{r} \left( \sum_{t=0}^{k_{a}-1} \zeta_{\ell^a}^{i_{a}s q^t} \right)
\end{equation}
for some $s \in \Z_{\ell^r}^\times$.  Since ${\rm{ord}}_{\ell^r}(q)=k_r$, multiplication by $q$ in $\Z_{\ell^r}^\times$ can be described by $\frac{\phi(\ell^r)}{k_r}$ disjoint $k_r$-cycles.  This implies that there are $\frac{\ell^{r-1}(\ell -1)}{k_r}$ sets of distinct exponents $\{s, sq, \dots, sq^{k_r -1} \} \pmod{\ell^r}$ in the last term as $s$ varies.  Thus, as we range over all $\tau \in G_{\ell^r}$, we see that
\begin{equation*}
\tau \left(\zeta_{\ell^r} + \zeta_{\ell^r}^{q} + \cdots + \zeta_{\ell^r}^{q^{k_r -1}} \right) 
\end{equation*}
takes $\frac{\ell^{r-1}(\ell -1)}{k_r}$ distinct values.  The products in equation \eqref{C} are all character values, so they are linearly independent.  This means that the character value $$\prod_{a=1}^{r} \left( \sum_{t=0}^{k_{a} -1} \zeta_{\ell^a}^{i_{a} q^{t}} \right)$$ has $\frac{\ell^{r-1}(\ell -1)}{k_r}$ distinct Galois conjugates, and so we have that $$\left[\Q\left(\prod_{a=1}^{r} \left( \sum_{t=0}^{k_{a} -1} \zeta_{\ell^a}^{i_{a} q^{t}} \right) \right) : \Q \right] = \frac{\ell^{r-1}(\ell -1)}{k_r}.$$  On the other hand, we observe that
\begin{equation*}
\left[ \Q(\zeta_{\ell^r})^{\langle \tau_{q} \rangle} : \Q \right] = \frac{\left[ \Q(\zeta_{\ell^r}) : \Q \right]}{\left[ \Q(\zeta_{\ell^r}): \Q(\zeta_{\ell^r})^{\langle \tau_{q} \rangle} \right]} = \frac{\ell^{r-1}(\ell -1)}{k_r}.
\end{equation*}
Since $G_{\ell^r} \cong \Z_{\ell^r}^\times$ is cyclic, there is a unique subfield of $\Q(\zeta_{\ell^r})$ of degree $d$ over $\Q$ for each $d\mid\phi\left(\ell^r\right)$.  Therefore, we must have that
\begin{equation*}
\Q\left( \prod_{a=1}^{r} \left( \sum_{t=0}^{k_{a} -1} \zeta_{\ell^a}^{i_{a} q^{t}} \right) \right) = \Q(\zeta_{\ell^r})^{\langle \tau_{q} \rangle}.
\end{equation*}
This proves Lemma \ref{Lemma3.1}.
\end{proof}


\begin{thebibliography}{1}


\bibitem{Bonnafe} C. Bonnaf\'e, \emph{Sur les caract\`eres des groupes r\'eductif finis
\`a centre non connexe: applications aux groupes sp\'eciaux lin\'eaires et unitaires},
Ast\'erisque {\bf 306}, 2006.

\bibitem{DOW} M. L. Dawsey, K. Ono, and I. Wagner.
\newblock Multiquadratic fields generated by characters of $A_n$.
\newblock \emph{J. Algebra} \textbf{533} (2019), 339-343.

\bibitem{FH} W. Fulton and J. Harris.
\newblock \emph{Representation theory: A first course}, Springer, New York, 1999.

\bibitem{G} J. A. Green.
\newblock The characters of the finite general linear groups.
\newblock \emph{Trans. Amer. Math. Soc.} \textbf{80}: 2 (1955), 402-447.

\bibitem{H} D. Hilbert.
\newblock Mathematical problems.
\newblock \emph{Bull. Amer. Math. Soc.} \textbf{8} (1902), 437-479.

\bibitem{RT} G. R. Robinson and J. G. Thompson.
\newblock Sums of squares and the fields $\Q_{A_n}$.
\newblock \emph{J. Algebra} \textbf{174} (1995), 225-228.

\bibitem{Shoji1} T. Shoji, \emph{Shintani descent for special linear groups},
J. Algebra \textbf{199} (1998), 175-228.

\bibitem{Shoji2} T. Shoji, \emph{Lusztig's conjecture for finite special linear groups},
Representation Th. \textbf{10} (2006), 164-222.


\bibitem{T} J. G. Thompson.
\newblock Personal letter to the second author, February 11, 1994.

\end{thebibliography}
\end{document}